\documentclass[letterpaper, 10 pt, conference]{ieeeconf}

\IEEEoverridecommandlockouts
\let\labelindent\relax
\usepackage{enumitem}       

\usepackage{amsmath,amssymb,amsfonts}
\usepackage{dsfont}
\usepackage[cal=cm, scr=boondox]{mathalpha}
\usepackage{graphicx}
\usepackage{mathtools}
\mathtoolsset{showonlyrefs}
\usepackage{ntheorem}
\usepackage{cite}

\allowdisplaybreaks

\newcommand{\E}{\mathbb{E}}
\newcommand{\R}{\mathbb{R}}
\newcommand{\N}{\mathbb{N}}
\newcommand{\Prob}{\mathbb{P}}

\newcommand{\be}{\mathbf{e}}
\newcommand{\mbf}{\mathbf{f}}

\newcommand{\bv}{\mathbf{v}}
\newcommand{\bw}{\mathbf{w}}
\newcommand{\bz}{\mathbf{z}}

\newcommand{\bA}{\mathbf{A}}
\newcommand{\bI}{\mathbf{I}}
\newcommand{\bJ}{\mathbf{J}}
\newcommand{\bK}{\mathbf{K}}
\newcommand{\bP}{\mathbf{P}}

\newcommand{\cC}{\mathcal{C}}
\newcommand{\cD}{\mathcal{D}}
\newcommand{\cE}{\mathcal{E}}
\newcommand{\cF}{\mathcal{F}}
\newcommand{\cG}{\mathcal{G}}
\newcommand{\cJ}{\mathcal{J}}
\newcommand{\cK}{\mathcal{K}}
\newcommand{\cM}{\mathcal{M}}
\newcommand{\cN}{\mathcal{N}}

\newcommand{\cS}{\mathcal{S}}
\newcommand{\cX}{\mathcal{X}}

\newcommand{\sI}{\mathscr{I}}
\newcommand{\sH}{\mathscr{H}}

\newcommand{\slo}{\underline{\sigma}}

\newcommand{\ind}[1]{\mathds{1}_{#1}}

\newtheorem{assumption}{Assumption}
\newtheorem{remark}{Remark}
\newtheorem{lemma}{Lemma}
\newtheorem{proposition}{Proposition}
\newtheorem{definition}{Definition}
\newtheorem{theorem}{Theorem}

\DeclareMathOperator*{\st}{s.t.}

\DeclareMathOperator*{\binomial}{Binomial}
\DeclareMathOperator*{\diag}{diag}
\DeclareMathOperator*{\lin}{lin}
\DeclareMathOperator*{\pw}{pw}

\newcommand{\Assum}{Assumption}
\newcommand{\Prop}{Proposition}
\newcommand{\Lemm}{Lemma}

\newcommand{\Rem}{Remark}
\newcommand{\Thm}{Theorem}
\newcommand{\Sect}{Section}
\newcommand{\Defn}{Definition}
\providecommand{\figurename}{Fig.}

\begin{document}


\title{\LARGE \bf
Training Dynamics and Induced Preconditioning \\ of ReLU Policy Gradient for Scalar LQR
}

\author{
    Jhojan A. Rodriguez-Gil and C\'esar A. Uribe
    \thanks{
        Jhojan A. Rodriguez-Gil and C\'esar A. Uribe are with the Department of Electrical and Computer Engineering, Rice University, Houston, TX 77005, USA. {\tt\small \{jr200, cauribe\}@rice.edu}
    }
}

\maketitle
\thispagestyle{empty}
\pagestyle{empty}

\begin{abstract}
Policy gradient is not invariant to reparameterization; the training dynamics of the weights of a neural controller differ from those of the induced gains. We make that difference exact for deterministic scalar discounted linear-quadratic regulation with bias-free one-hidden-layer ReLU policies. We show that policy gradient on the weights induces preconditioning on the gains' dynamics, alongside an explicit quadratic finite-step correction and an activation-group error. Moreover, we identify explicit conditions on the step sizes and initialization parameters that guarantee that the sequence of generated controllers remains stabilizing throughout training and converges to the optimal gains with high probability. For a given system and cost parameters, the sufficient step-size conditions imply contraction bounds independent of the network width. Numerical analysis illustrates how the neural realization changes the induced gain dynamics while maintaining stability across training.
\end{abstract}

\section{INTRODUCTION}

Control policies are increasingly trained rather than designed: a controller is parameterized, a cost is differentiated, and gradient descent is run in the parameter space. This connects back to classical theory, where the linear-quadratic regulator (LQR) has provided a well-understood setting for directly parameterized linear policies. Its gradient dominance supports global convergence under suitable initialization and step-size conditions~\cite{fazel2018global,hu2023toward}, with established sample-complexity bounds for model-free optimization~\cite{malik2019derivative,moghaddam2025sample}. In this direct parameterization, ordinary policy gradient acts on the gains themselves, and natural-gradient and Gauss--Newton methods can add explicit preconditioning to that update~\cite{fazel2018global}. Importantly, policy-gradient theory has linked policy-class structure to global optimality guarantees~\cite {agarwal2021theory,bhandari2024global}. 

When using neural reparameterization, the effective policy update changes even when the optimization algorithm remains gradient descent. In particular, it induces a weight-dependent preconditioner, as characterized for deep linear factorizations and, more generally, through neural tangent kernels~\cite{arora2018optimization,jacot2018neural}. Homogeneous networks also satisfy balance identities under gradient flow~\cite{du2018algorithmic}. Convergence guarantees for overparameterized networks are also available for supervised regression with smooth~\cite{du2019gradient} and ReLU activations~\cite{du2019gradient_descent}.

For neural controllers, these optimization questions are coupled with the requirement of closed-loop stability throughout training. For continuous-time LQR, gradient-flow analyses of linear feedforward controllers have established stability and convergence properties~\cite{oliveira2025overparamlqr}, including almost-global exponential convergence in simplified settings~\cite{wafi2026almost}. Neural policy-gradient convergence has also been studied in overparameterized actor--critic settings~\cite{wang2020neuralpg}, and other parameterizations enforce closed-loop stability by construction~\cite{furieri2022neural}. These results do not directly establish whether finite gradient steps in ReLU controller weights preserve stability and converge to the optimal LQR policy as neurons change activation groups. \textit{We address this question} for deterministic scalar discounted LQR with a residual, bias-free, one-hidden-layer ReLU controller, where the induced policy is exactly characterized by two effective gains, one on each half-line, which allows an exact description of the gain updates induced by training the neural weights. 

Let $\bK_{\theta_k}=(K_{1,\theta_k},K_{2,\theta_k})$ denote these gains and $\cJ$ the policy cost. A policy gradient step in the weights with step size $\eta$ induces the following gain update, shown alongside direct gradient descent on the same two-gain objective:
\begin{align*}
\text{\small (Weights)} \quad
 \theta_{k+1}
  &= \theta_k-\eta\nabla_{\theta}\cJ(\bK_{\theta_k}), \\
\text{\small (Induced gains)} \quad
 \!\!\!\bK_{\theta_{k+1}}
  &= \bK_{\theta_k}
  {-}\eta D_k\nabla_{\bK}\cJ(\bK_{\theta_k})
  {+}\eta^2h_k{+}\be_k, \\
\text{\small (Direct gains)} \quad
 \bK_{k+1}
 &= \bK_k-\eta\nabla_{\bK}\cJ(\bK_k).
\end{align*}
Here, $D_k$ is a preconditioner generated by the neural parameterization itself: it depends on the weights realizing the controller, rather than on the effective gains alone. The term $\eta^2h_k$ is the bilinear finite-step correction, and $\be_k$ accounts for neurons changing activation groups. Our analysis controls the evolving preconditioner and both correction terms to establish cost descent, closed-loop stability, and convergence. Our contributions are:
\begin{itemize}[leftmargin=2.0em,left=-1pt,itemsep=1pt, parsep=-1pt, topsep=-0pt, partopsep=-1pt]
\item We establish gradient dominance and quadratic growth for the two-gain objective on a stability-margin set, with explicit constants that convert descent of the induced gain dynamics into convergence in cost and gains.
\item We derive the exact gain dynamics induced by ordinary gradient descent in ReLU weights. These dynamics reveal a realization-dependent preconditioner generated by the neural parameterization itself, together with bilinear finite-step corrections and activation-group terms.
\item We identify a nonempty initialization regime in which, for sufficiently large width and admissible step sizes, training preserves closed-loop stability and converges geometrically to the optimal cost and gains with high probability. For a given problem and initialization scale, the sufficient step-size and contraction bounds are independent of width.
\end{itemize}

\noindent \Sect~\ref{section_problem_formulation} states the problem and main result; \Sect~\ref{section_policy_level} and \Sect~\ref{section_weights_level} develop the controller-level geometry and induced training dynamics; \Sect~\ref{section_proof_main_result} proves the theorem; and \Sect~\ref{section_simulations} presents the numerical study.

\section{PROBLEM FORMULATION AND MAIN RESULT}
\label{section_problem_formulation}

Consider the deterministic scalar discrete-time system

\vspace{-0.3cm}
\begin{equation} \label{eq_dynamics}
    x_{t+1} = a x_t + b u_t, \quad t=0,1,2,\dots,
\end{equation}
\vspace{-0.4cm}

\noindent where $x_t \in \R$ is the state, $u_t \in \R$ is the input, and $a,b \in \R$ are the system parameters. Given an initial state $x_0 \sim \cD$, we study the discounted infinite-horizon regulation problem
\begin{equation} \label{eq_org_problem}
    \begin{array}{cl}
        \underset{u_0,u_1,\ldots}{\min}
        &
        \E_{x_{0} \sim \cD} \left[\sum_{t=0}^{\infty} \gamma^{t} \bigl( qx_{t}^{2} + ru_{t}^{2} \bigr) \right]
        \\
        \st & x_{t+1} = ax_{t} + bu_{t}, \qquad t=0,1,2,\ldots
    \end{array}
\end{equation}
with $q>0$, $r>0$, and $\gamma\in(0,1)$. For a  policy $\mu$, define

\vspace{-0.5cm}
\[
J_\mu(x)\triangleq\sum_{t=0}^{\infty}\gamma^t
\bigl(qx_t^2+r\mu(x_t)^2\bigr),
\qquad x_0=x,
\]
\vspace{-0.3cm}

\noindent allowing the value $+\infty$. The feedback-design problem is
$\min_{\mu\in\cF}\E_{x_0\sim\cD}[J_\mu(x_0)]$, where $\cF$ is the class of stabilizing deterministic state-feedback laws. We drop the subscript on~$J_\mu$ when the policy is fixed. 

The unconstrained discounted LQR problem has a unique statewise optimal linear policy $\mu^\star(x)=K^\star x$~\cite{anderson1989optimal}; \Assum~\ref{assumption_controllability} below ensures that this policy belongs to $\cF$.
The discounted Riccati coefficient $P^{\star}>0$ and gain are characterized by
\begin{equation}\label{eq_riccati}
    P^{\star}=q+\frac{\gamma a^{2}rP^{\star}}{r+\gamma b^{2}P^{\star}},
    \qquad K^{\star}=-\frac{\gamma abP^{\star}}{r+\gamma b^{2}P^{\star}}.
\end{equation}

\begin{assumption} \label{assumption_controllability}
    System \eqref{eq_dynamics} is controllable, i.e., $b \neq 0$; we have access to a stabilizing controller $K_{0}$, i.e., $| a + b K_{0} | < 1$; and the optimal policy $K^{\star}$ is stabilizing, i.e., $|a + b K^{\star}| < 1$.
\end{assumption}
 
\begin{assumption} \label{assumption_initial_distribution}
    The moments
    $\sigma_{1}^{2} \triangleq \E[x_{0}^{2}\ind{\{x_{0} \ge 0\}}]$ and
    $\sigma_{2}^{2} \triangleq \E[x_{0}^{2}\ind{\{x_{0} < 0\}}]$
    are finite and strictly positive. We write
    $\slo^{2} {\triangleq} \min\{\sigma_{1}^{2}, \sigma_{2}^{2}\} {>} 0$ and
    $\sigma_{\cD}^{2} {\triangleq} \sigma_{1}^{2} + \sigma_{2}^{2} {=} \E[x_{0}^{2}]$.
\end{assumption}
 
\Assum~\ref{assumption_controllability} guarantees that the set of stabilizing policies is nonempty and that we have access to one of its elements. The assumption on $K^{\star}$ is separate from the finiteness of the discounted cost, which for a linear gain requires only $\gamma|a+bK|^{2}<1$. It also lets us define the \emph{common baseline and optimum margin}
\begin{equation} \label{eq_delta_zero}
    \delta_{0} \triangleq \min\bigl\{ 1 - |a + bK_{0}|,\; 1 - |a + bK^{\star}| \bigr\} > 0 .
\end{equation}
\Assum~\ref{assumption_initial_distribution} is a moment condition used to lower bound the positive- and negative-side occupancy weights; it requires finite second moments and positive mass away from zero on both half-lines. 

We define the control law $\mu_{\theta} (x) = K_{0} x + \widetilde{\mu}_{\theta}(x)$, where $\widetilde{\mu}_{\theta}(x) \triangleq \widetilde{\mu}(\theta, x)$ is a bias-free one-hidden-layer ReLU neural network with $m$ neurons, i.e.,
\begin{equation} \label{eq_neural_cntr_def}
    \widetilde\mu_{\theta}(x) {=} \sum_{i=1}^{m} v_i \sigma(w_i x), \quad \sigma(y) {=} \max\{0,y\}, \quad \theta {\triangleq} [\bw^{\top}\; \bv^{\top}]^{\top}
\end{equation}
where $\bw, \bv \in \R^{m}$ are the input- and output-layer weights. Let $\sI=\{1,\dots,m\}$ be the set of all neuron indices, and $\sI_1\triangleq\{j\in\sI:w_j > 0\}$ and $\sI_2\triangleq\{j\in\sI:w_j<0\}$. Then, equivalently
\begin{equation}
    \mu_{\theta}(x) =
    \begin{cases}
        K_{1,\theta} x, & x \ge 0, \\
        K_{2,\theta} x, & x < 0,
    \end{cases}
    \label{eq_piecewise_controller}
\end{equation}
where $K_{i,\theta} {\triangleq} K_{0} {+} \widetilde{K}_{i,\theta}$ with $\widetilde{K}_{i,\theta} {\triangleq} \sum_{j \in \sI_{i}} v_{j} w_{j}$, $i \in \{1,2\}$. Thus, the ReLU controller is completely characterized by the two effective gains $(K_{1,\theta},\; K_{2,\theta}) \triangleq \bK_{\theta}= \bK_{\mathrm b} + \widetilde{\bK}_{\theta}$, with $\widetilde{\bK}_{\theta} \triangleq (\widetilde{K}_{1,\theta}, \widetilde{K}_{2,\theta})$ and $\bK_{\mathrm b} \triangleq (K_{0}, K_{0})$.
 
Write $z_{+}\triangleq\max\{z,0\}$ and $z_{-}\triangleq\max\{-z,0\}$. For any $\delta \in (0,1)$, define the stability-margin set
$
    \cK_{\delta} \triangleq \{ \bK : |a+bK_{i}| \le 1-\delta,\; i \in \{1, 2\} \}
$, and the safe set $\cK_{\overline{\delta}}$ with $\overline{\delta} \triangleq \delta/2$. Since each constraint restricts a single coordinate, $\cK_{\delta}$ is a compact rectangle, hence convex, and $\cK_{\delta} \subset \cK_{\overline{\delta}}$. Moreover, if $\delta < \delta_{0}$ then both $\bK_{\mathrm b}$ and $\bK^{\star} \triangleq (K^{\star},K^{\star})$ belong to $\cK_{\delta}$. The algorithm itself is unconstrained; $\cK_{\overline{\delta}}$ is the region whose invariance we prove. Membership in this box gives $|x_{t+1}|\le(1-\overline\delta)|x_t|$. The box is a sufficient stability region, not the entire set of stabilizing two-gain controllers. 

The parameterized policy-search problem is
\begin{equation} \label{eq_rl_theta_problem}
    \underset{\theta}{\min}
     \ \cJ(\theta),
    \;\;
    \st \;  x_{t+1} = ax_{t} + b\mu_{\theta}(x_{t}),
\end{equation}
where $\cJ(\theta) \triangleq \E_{x_{0} \sim \cD}[J(x_0)]$ is the policy cost; since $\bK_{\theta}$ determines the closed loop, $\cJ(\theta) = \cJ(\bK_{\theta})$ and we use the two notations interchangeably. Outside the finite-cost domain, $\cJ(\theta)$ is understood to be $+\infty$; we show that each gradient evaluation along the initialized trajectory lies within the finite-cost safe region. For $m\ge2$, the optimal linear controller is representable by the residual ReLU parameterization; thus, the main issue is not expressivity but optimization: policy gradient operates in the redundant and nonconvex weight space $\theta$. We use exact model-based gradients and the convention $\sigma'(0)=0$ in backpropagation. 
 
\begin{theorem}\label{theorem_main_result}
    Let Assumptions~\ref{assumption_controllability}--\ref{assumption_initial_distribution} hold and fix a stability margin $\delta \in (0,\delta_{0})$ together with constants $0 < s < S < \infty$, $V > 0$ and $\varepsilon \in (0,1)$. There exist $\beta_{0} > 0$, maps $m_{0}(\cdot)\ge2$ and $C_{0}(\cdot)>0$ such that the following holds for every $\beta \ge \beta_{0}$, every width $m \ge m_{0}(\beta)$, and every admissible step size $0<\eta\le\eta_0(\beta)$ (\Defn~\ref{definition_admissible_step_sizes}). Initialize independently $w_{j,0} \sim \cN(0,{\beta}/{m})$ and $v_{j,0} \sim \cN(0,{\beta^{-2}}/{m})$, $j=1,\dots,m$. Then, with probability at least $1 - e^{-C_{0}(\beta)m}$ over the initialization, the policy-gradient iterates
    \begin{equation}\label{eq_update_rule}
        \theta_{k+1} = \theta_k - \eta\, \nabla_{\theta} \cJ(\bK_{\theta_{k}})
    \end{equation}
    satisfy, for every $k \ge 0$, writing $\Delta_{k} \triangleq \cJ(\bK_{\theta_{k}}) - \cJ(\bK^{\star})$:
    \begin{enumerate}[label=(\roman*),leftmargin=1.7em,itemsep=1pt,topsep=2pt,parsep=0pt]
        \item (Stability) $\bK_{\theta_{k}} \in \cK_{\overline{\delta}}$;
        \item (Geometric convergence)
        \begin{equation}
            \Delta_{k} \le \varrho^{k} \Delta_{0},
            \quad
            \|\bK_{\theta_{k}} {-} \bK^{\star}\|^2 \le \frac{\varrho^{k} \Delta_{0}}{(r {+} \gamma b^{2} P^\star)\slo^{2}}.
        \end{equation}        
    \end{enumerate}
    The contraction factor is $\varrho \triangleq 1 - \eta\, \lambda_{\ast}\, \alpha_{\bK} \in (0,1)$, where $\alpha_{\bK}$ is the gradient-dominance constant of \Lemm~\ref{lemma_J_geometry}(iii) and
    \begin{equation} \label{eq_lambda_star_thm}
        \lambda_{\ast}
        =
        (1 {-} \varepsilon) \bigl(\Phi(S) {-} \Phi(s)\bigr) \bigl(2 \Phi(V \beta) {-} 1\bigr) {s^{2} \beta}/{4},
    \end{equation}
    and $\Phi$ is the standard Gaussian CDF.
\end{theorem}
\Thm~\ref{theorem_main_result} implies that, starting from a stabilizing baseline controller, overparameterization together with an admissible step size yields a sequence of stable residual ReLU controllers whose cost decreases geometrically toward the optimal LQR cost. The admissible parameters and the proof are given in \Sect~\ref{section_proof_main_result}.
 
\begin{remark}\label{remark_width_free}
All deterministic constants entering the admissible step size and contraction factor depend on the system and cost parameters, initialization scale, and auxiliary parameters, but not on the width~$m$. In particular,
$\E[\sum_j(w_{j,0}^2+v_{j,0}^2)]=\beta+\beta^{-2}$, and on the initialization event. Thus, the sufficient step size and contraction bounds do not deteriorate with width. The mass fluctuations, actual trajectories, and per-step computations may still depend on~$m$. Increasing $m$ improves the certified success-probability lower bound $1-e^{-C_0(\beta)m}$.
\end{remark}

\section{CONTROLLER-LEVEL LANDSCAPE}
\label{section_policy_level}

We first study the geometry of $\cJ$ under~\eqref{eq_piecewise_controller}. Throughout this section we work directly with $\bK = (K_{1}, K_{2})$; \Sect~\ref{section_weights_level} extends the analysis to the influence of~$\theta$.
 
\begin{lemma} \label{lemma_state_value_function}
   Let Assumptions~\ref{assumption_controllability}--\ref{assumption_initial_distribution} hold and let $\delta \in (0, 1)$. If $\bK\in \cK_{\overline{\delta}}$, then, writing $a_i\triangleq a+bK_i$, $J(x_{0}) = P_{1} x_{0}^2 \ind{\{x_0 \ge 0\}} + P_{2} x_{0}^2 \ind{\{x_0 < 0\}}$ where $P_{1}, P_{2}$ are uniquely defined as the solution of $\bP(\bK) = (\bI - \gamma \bA(\bK) )^{-1} \mbf(\bK)$, with
    {\small
    \begin{equation}
        \bP(\bK) {\triangleq}
        \begin{bmatrix}
            P_{1} \\
            P_{2}
        \end{bmatrix},\;\;
        \mbf(\bK) {\triangleq}
        \begin{bmatrix}
            q + rK_{1}^2 \\
            q + rK_{2}^2
        \end{bmatrix},\;\;
        \bA(\bK) {\triangleq}
        \begin{bmatrix}
            (a_{1})_{+}^2 & (a_{1})_{-}^2 \\
            (a_{2})_{-}^2 & (a_{2})_{+}^2
        \end{bmatrix}.
    \end{equation}
    }Moreover, $\cJ(\bK) = \sigma_{1}^{2} P_{1} + \sigma_{2}^{2} P_{2}$, $P_{1}, P_{2} \ge q$, and $\max(P_{1}, P_{2}) \le P_{\max} \triangleq \varphi_{\max} / (1 - \gamma(1 - \overline{\delta})^2)$ with $K_{\max} \triangleq (|a| + 1 - \overline{\delta}) / |b|$ and $\varphi_{\max} \triangleq q + rK_{\max}^2$.
\end{lemma}
 
\begin{proof}
    The Bellman equations on the two half-lines are
    $P_i=q+rK_i^2+\gamma a_i^2P_{n(i)}$, where $n(i)=i$ if $a_i\ge0$ and $n(i)= 3 -i$ otherwise. Thus $\bP=\mbf+\gamma\bA\bP$. Drop the argument $\bK$ from $\bA$, $\bP$, $\mbf$. Each
    row of $\bA$ has at most one nonzero entry, so on $\cK_{\overline{\delta}}$
    \begin{equation} \label{eq_neumann}
    \|\bA\|_{\infty}{\le}(1{-}\overline\delta)^2,
        \ \|\gamma\bA\|_{\infty}{<}1, \ (\bI{-}\gamma\bA)^{{-}1}{=}\sum_{t\ge0}(\gamma\bA)^t,
    \end{equation}
    which gives existence and uniqueness of $\bP$. Taking expectations in $J(x_{0})$ and using
    \Assum~\ref{assumption_initial_distribution} gives
    $\cJ(\bK) = \sigma_{1}^{2}P_{1} + \sigma_{2}^{2}P_{2}$. Since every $(\gamma\bA)^{t}$ is
    entrywise nonnegative, \eqref{eq_neumann} turns the componentwise bounds
    $q\mathbf{1} \le \mbf \le \varphi_{\max}\mathbf{1}$ into $P_{i} \ge q$ and
    $\|\bP\|_{\infty} \le \varphi_{\max}\sum_{t\ge0}\gamma^{t}(1{-}\overline{\delta})^{2t} = P_{\max}$.
\end{proof}
 
Statewise optimality of the Riccati controller also gives $P_i(\bK)\ge P^{\star}$ for every $\bK\in\cK_{\overline\delta}$. Write $R^{\star}\triangleq r+\gamma b^{2}P^{\star}$ and $d_{0}\triangleq1-\gamma(1-\overline\delta)^{2}$, and define the unnormalized discounted occupancy moments $\mathbf d(\bK)\triangleq ((\bI-\gamma\bA(\bK))^{-1})^{\top}
       [\sigma_1^2,\sigma_2^2]^{\top}$.
Equivalently, with $\cX_1=[0,\infty)$ and $\cX_2=(-\infty,0)$,

\vspace{-0.3cm}
\[
d_i(\bK)=\sum_{t=0}^{\infty}\gamma^t
 \E\!\left[x_t^2\ind{\{x_t\in\cX_i\}}\right],
\]
\vspace{-0.3cm}

\noindent where the trajectory is generated by $\bK$. Indeed, the vector of half-line second moments evolves by multiplication by $\bA(\bK)^\top$.
Equation~\eqref{eq_neumann} gives $\sigma_i^2\le d_i(\bK)\le\sigma_{\cD}^2/d_0$. Differentiating $\bP=\mbf+\gamma\bA\bP$ and contracting with $[\sigma_1^2,\sigma_2^2]^{\top}$ writes each gradient component as a displacement from an explicit target,
\vspace{-0.2cm}
\begin{equation}\label{eq_occupancy_gradient}
    \begin{gathered}
    g_i\triangleq\partial_{K_i}\cJ(\bK)
      =2d_i(\bK)\bigl(rK_i+\gamma b\,a_iP_{n(i)}\bigr)
      \\
      =2d_i(\bK)R_i(\bK)\bigl(K_i-\widehat K_i(\bK)\bigr),
    \end{gathered}
\end{equation}
with $n(i)$ as in \Lemm~\ref{lemma_state_value_function}, $R_i(\bK)\triangleq r+\gamma b^2P_{n(i)}$ and $\widehat K_i(\bK)\triangleq-\gamma abP_{n(i)}/R_i(\bK)$. Since $P_{n(i)}\ge P^{\star}$, the targets satisfy $|a+b\widehat K_i(\bK)|=|a|r/(r+\gamma b^2P_{n(i)})\le|a+bK^{\star}|$: each coordinate of $-\nabla_{\bK}\cJ$ points toward a gain at least as stable as $K^{\star}$, and this inward structure, together with finite-step and crossing bounds, keeps the iterates in $\cK_{\overline{\delta}}$ in \Sect~\ref{section_proof_main_result}.

\begin{lemma}\label{lemma_J_geometry}
    Let Assumptions~\ref{assumption_controllability}--\ref{assumption_initial_distribution} hold, let $\delta\in(0,\delta_0)$, and set $\bK^\star \triangleq (K^{\star}, K^{\star})$, where $P^{\star} > 0$ is the corresponding Riccati coefficient, so that $\bK^{\star} \in \cK_{\overline{\delta}}$. Then:
 
    (i) (Smoothness) There exists $L_\bK<\infty$ such that
    \begin{equation}
        \| \nabla_{\bK} \cJ(\bK) {-} \nabla_{\bK} \cJ(\bK')\|
        {\le}
        L_\bK \|\bK {-} \bK'\|,
        \; \forall\, \bK, \bK' {\in} \cK_{\overline{\delta}}.
    \end{equation}
 
    (ii) (Quadratic growth) For all $\bK \in \cK_{\overline{\delta}}$,
    \begin{equation}
        \cJ(\bK)-\cJ(\bK^\star)
        \ge
        (r+\gamma b^2P^\star)\,\slo^{2} \,\|\bK - \bK^\star\|^2.
    \end{equation}
 
    (iii) (Gradient dominance) Let $R^{\star}\triangleq r+\gamma b^2P^{\star}$ and $\mathbf d^{\star}\triangleq\mathbf d(\bK^{\star})$. Then, it holds that
      \begin{equation}
        \tfrac{1}{2} \| \nabla_{\bK} \cJ(\bK)\|^2
        \ge
        \alpha_{\bK} \bigl( \cJ(\bK) - \cJ(\bK^\star) \bigr),
        \quad
        \forall \bK \in \cK_{\overline{\delta}},
    \end{equation}
    where
    \begin{equation}\label{eq_explicit_pl}
        \alpha_{\bK}\triangleq
        \frac{2R^{\star}}{\max_{i=1,2}d_i^{\star}/(\sigma_i^2)^2}>0.
    \end{equation}

    (iv) (Uniform gradient bound) With $d_0 \triangleq 1 - \gamma(1 - \overline{\delta})^{2}$,

    \vspace{-0.2cm}
    \begin{equation}
        \|\nabla_{\bK} \cJ(\bK)\| {\le} G_{\max}
        {\triangleq}
        \frac{2\sqrt{2}\,\sigma_{\cD}^2 \bigl(rK_{\max}+\gamma |b|(1-\overline\delta)P_{\max}\bigr)}{ d_0}.
    \end{equation}
    \vspace{-0.2cm}
 
    Consequently, with $C_{\mathrm{grad}} \triangleq {L_\bK^2}/{((r{+}\gamma b^2P^\star)\slo^{2})}$,
    \begin{equation} \label{eq_cgrad}
        \| \nabla_{\bK} \cJ(\bK)\|^2
        \le
        C_{\mathrm{grad}} \bigl( \cJ(\bK) - \cJ(\bK^\star) \bigr),
        \quad \forall \bK \in \cK_{\overline{\delta}}.
    \end{equation}
\end{lemma}
 
\begin{proof}
\emph{(i) Smoothness.} By \Lemm~\ref{lemma_state_value_function}, $\cJ(\bK)=\sigma_{1}^{2}P_{1}(\bK) + \sigma_{2}^{2}P_{2}(\bK)$, and since $(\cdot)_+^2$ and $(\cdot)_-^2$ have Lipschitz derivatives and $\|\bA(\bK)\|_\infty\le (1 - \overline{\delta})^2$ on $\cK_{\overline\delta}$, the map $\bK \mapsto \bP(\bK)$ is $\mathcal{C}^{1,1}$ on $\cK_{\overline{\delta}}$. Hence $\nabla_{\bK} \cJ$ is Lipschitz on $\cK_{\overline{\delta}}$.
Writing $\partial_i=\partial_{K_i}$, an explicit, conservative choice is obtained by setting
\vspace{-0.1cm}
\[
\begin{aligned}
\kappa&=1-\overline\delta, \ T={2(rK_{\max}+\gamma|b|\kappa P_{\max})}/{d_0},\\
H&=\frac{2r+2\gamma b^2P_{\max}+4\gamma|b|\kappa T}{d_0}, \
L_{\bK}{=}\max\{\alpha_{\bK},\,2\sigma_{\cD}^2H\}.
\end{aligned}
\]
Differentiation of the value equations gives the following:
$\|\partial_i\bP\|_\infty\le T$ and
$\|\partial_i\partial_j\bP\|_\infty\le H$ on each sign cell.
The matching first derivatives across cell boundaries give the same
Lipschitz bound on the entire safe box.

\emph{(ii) Quadratic growth.} Completing the square in~\eqref{eq_riccati} gives
\[
(\bI-\gamma\bA)(\bP-P^\star\mathbf1)
 =R^\star[(K_1-K^\star)^2,(K_2-K^\star)^2]^\top.
\]
Multiplying by $[\sigma_1^2,\sigma_2^2](\bI-\gamma\bA)^{-1}$ yields
\vspace{-0.1cm}
\begin{equation}\label{eq_exact_gap}
\cJ(\bK)-\cJ(\bK^\star)
 =R^\star\sum_{i=1}^2d_i(\bK)(K_i-K^\star)^2.
\end{equation}
\vspace{-0.3cm}

\noindent Since $d_i(\bK)\ge\sigma_i^2\ge\slo^2$, this proves (ii).

\emph{(iii) Gradient dominance.} Let $g=\nabla_{\bK}\cJ(\bK)$ and define $Q(x,u) \triangleq qx^{2}+ru^{2}+\gamma J(ax+bu)$. For a nonzero $x$ on half-line $i$, define $\phi_i(z)\triangleq Q(x,zx)/x^2$, which does not depend on the magnitude of $x$; since $P_1, P_2\ge P^{\star}$, its piecewise second derivatives are at least $2R^{\star}$ and its first derivatives match at the breakpoint, so it is $2R^{\star}$-strongly convex, and~\eqref{eq_occupancy_gradient} reads $g_i=d_i\phi_i'(K_i)$. Telescoping $J_{\bK}-J_{\bK^{\star}}$ along the optimal trajectory, with $ \Delta\triangleq\cJ(\bK)-\cJ(\bK^{\star})$ gives
\begin{equation}
    \begin{aligned}
    \Delta&=\sum_{i=1}^2 d_i^{\star}\bigl[\phi_i(K_i)-\phi_i(K^{\star})\bigr] \le\sum_{i=1}^2\frac{d_i^{\star}}{4R^{\star}d_i^2}\,g_i^2
       \le\frac{\|g\|^2}{2\alpha_{\bK}}.
    \end{aligned}
\end{equation}

Strong convexity gives $\phi_i(K_i) -\phi_i(K^{\star}) \le \phi_i'(K_i)^2 / (4R^{\star})$, and $d_i\ge\sigma_i^2$ gives the last bound. The telescoping remainder vanishes because the optimal state contracts and $J_{\bK}(x)\le P_{\max}x^2$.  We choose the smoothness constant $L_{\bK}$ large enough ($L_{\bK}\ge\alpha_{\bK}$); increasing a valid Lipschitz constant preserves (i).

\emph{(iv) Gradient bound.} Drop the argument $\bK$ and write $\partial_{i}$ for $\partial_{K_{i}}$. On $\cK_{\overline{\delta}}$ the four ingredients are
\begin{equation} \label{eq_grad_ingredients}
    \begin{gathered}
    \|(\bI{-}\gamma\bA)^{-1}\|_{\infty} \le d_{0}^{-1},
    \qquad
    \|\bP\|_{\infty} \le P_{\max},
    \\
    \|\partial_{i}\mbf\|_{\infty} \le 2rK_{\max},
    \qquad
    \|\partial_{i}\bA\|_{\infty} \le 2|b|(1{-}\overline{\delta}),
    \end{gathered}
\end{equation}
the second from \Lemm~\ref{lemma_state_value_function} and the rest from $|K_{i}| \le K_{\max}$. Differentiating $\bP = (\bI{-}\gamma\bA)^{-1}\mbf$ gives $\partial_{i}\bP = (\bI{-}\gamma\bA)^{-1}\bigl(\partial_{i}\mbf + \gamma(\partial_{i}\bA)\bP\bigr)$, so \eqref{eq_grad_ingredients} yields $\|\partial_{i}\bP\|_{\infty} \le (2/d_{0})\bigl(rK_{\max} + \gamma|b|(1{-}\overline{\delta})P_{\max}\bigr)$. Since $\cJ = \sigma_{1}^{2}P_{1} + \sigma_{2}^{2}P_{2}$, we get $|\partial_{i}\cJ| \le \sigma_{\cD}^{2}\|\partial_{i}\bP\|_{\infty}$, and the Euclidean norm over the two components gives $G_{\max}$. Finally, $\nabla_{\bK}\cJ(\bK^{\star}) {=} 0$ and $\bK^{\star}{\in} \cK_{\overline{\delta}}$, so (i) and (ii) give~\eqref{eq_cgrad}.
\end{proof}

\section{WEIGHT-SPACE GEOMETRY AND INDUCED TRAINING DYNAMICS}
\label{section_weights_level}

We now study the landscape induced by the map $\theta \mapsto (\widetilde{K}_{1,\theta}, \widetilde{K}_{2,\theta})$. We first collect the properties of the initialization required by the analysis. Define the sets of \emph{potentially problematic} neurons and their masses,
\begin{align} \label{eq_bad_sets}
        &\cS_{w} {\triangleq} \{ j {\in} \sI : |w_{j,0}| {\le} \varpi/\sqrt{m}\},\,
        \cS_{v} {\triangleq} \{ j {\in} \sI : |v_{j,0}| {>} V/\sqrt{m}\},
         \nonumber \\
        &\cM_{ \{w,v\}} {\triangleq} {\textstyle\sum_{j}} (w_{j,0}^{2} {+} v_{j,0}^{2})\ind{\{j\in \cS_{ \{w,v\}}\}},
\end{align}
with fixed $\varpi>0$ as defined in \Lemm~\ref{lemma_sign_stability} (with $B$ fixed in \Sect~\ref{section_proof_main_result}). The set $\cS_{w}$ contains the neurons near the boundary $w_{j}=0$ that separates the two residual gains, and $\cS_{v}$ the neurons whose second-layer weight is large; the latter controls the gradient of neuron $j$ with respect to $w$. Thus, for thresholds $M_w, M_v{>}0$, define $\cE_{\Xi} {\triangleq} \{\cM_{w} {\le} M_{w},\ \cM_{v} {\le} M_{v}\}$.
 
\begin{lemma}\label{lemma_initialization}
    Let \Assum~\ref{assumption_controllability} hold and let $\delta \in (0,\delta_{0})$. Fix $\alpha,\beta>0$ and initialize all weights mutually independently as $w_{j,0} \sim \cN(0,\beta/m)$ and $v_{j,0}\sim\cN(0,\alpha/m)$, $j = 1,\dots,m$. Fix $\varepsilon \in (0,1)$, $0 < \tau < W < \infty$ and $V>0$, set $R_{\ast}^{2} \triangleq (1+\varepsilon)(\alpha+\beta)$ and
    $
        p \triangleq ( \Phi ( {W}/{\sqrt{\beta}} ) - \Phi ( {\tau}/{\sqrt{\beta}} ) ) (2 \Phi ( {V}/{\sqrt{\alpha}} ) - 1 )
    $,
    and let $c \in (0,p)$. Define the backbone sets
    \begin{equation} \label{eq_backbone}
        \begin{gathered}
        \cC_1 {\triangleq} \big\{ j \in \sI : {\tau}/{\sqrt{m}} \le w_{j, 0} \le {W}/{\sqrt{m}}, \; |v_{j, 0} | \le {V}/{\sqrt{m}} \big\},
        \\
        \cC_2 {\triangleq} \big\{ j \in \sI : {-W}/{\sqrt{m}} {\le} w_{j, 0} {\le} {-\tau}/{\sqrt{m}},\, |v_{j, 0}| {\le} {V}/{\sqrt{m}} \big\},
        \end{gathered}
    \end{equation}
    and let $M_{w} > \E[\cM_{w}]$ and $M_{v} > \E[\cM_{v}]$. Then there exist $m_{\cE}(\beta) \in \N$ and $C_{0}(\beta)>0$, both independent of $m$, such that for every $m \ge m_{\cE}(\beta)$ the event
\begin{equation}\label{eq_good_event}
\begin{aligned}
\cE_m\triangleq{}&\cE_\Xi\cap
\bigl\{\bK_{\theta_0}\in\cK_\delta,\;
 |\cC_1|\ge cm,\;|\cC_2|\ge cm,\\
&\hspace{2em}\textstyle\sum_j(w_{j,0}^2+v_{j,0}^2)\le R_\ast^2\bigr\}
\end{aligned}
\end{equation}
satisfies $\Prob(\cE_{m})\ge 1-e^{-C_{0}(\beta) m}$.
\end{lemma}
 
\emph{Proof sketch.} Write
$w_{j,0}=\sqrt{\beta/m}\,Z_j$ and
$v_{j,0}=\sqrt{\alpha/m}\,U_j$, where the pairs $(Z_j,U_j)$ are i.i.d.\ standard Gaussian.
Each initial residual gain is an average of centered sub-exponential
variables $\sqrt{\alpha\beta}\,Z_jU_j\ind{\{Z_j\gtrless0\}}$.
Bernstein's inequality bounds its absolute value by
$(\delta_0-\delta)/|b|$ except with probability $C_{\mathrm g}e^{-c_{\mathrm g}m}$ for constants $C_{\mathrm g},c_{\mathrm g}>0$ independent of $m$,
which implies $\bK_{\theta_0}\in\cK_\delta$.
The total and bad-set masses are averages of sub-exponential variables
bounded by $\beta Z_j^2+\alpha U_j^2$; their deterministic thresholds
strictly exceed their expectations. Their failure probabilities are
therefore also exponentially small in $m$~\cite{vershynin2018high}.
Each backbone cardinality is $\binomial(m,p)$, so Chernoff's inequality
applies because $c<p$. A union bound gives $C'e^{-c'm}$, which is at most
$e^{-C_0m}$ for all sufficiently large $m$.

Note that $C_0$ and $m_{\cE}$ depend on $\beta$ and may deteriorate as $\beta$ grows. On $\cE_m$, \Lemm~\ref{lemma_initialization} also gives a stable initial residual controller, a bound on the total initial weight mass, and backbone sets $\cC_1,\cC_2$ of the stated sizes that depend on $\bw_0$ and therefore remain fixed throughout training.
 
\begin{remark} \label{remark_relu_subgradient}
    We adopt the convention $\sigma'(0)=0$. Under the initialization of \Lemm~\ref{lemma_initialization}, $w_{j,0} \neq 0$ almost surely for every~$j$, so zero-valued first-layer weights can only arise during training; such neurons with $w=0$ will remain frozen.
\end{remark}
 
We now study the weight-space updates \eqref{eq_update_rule}. By \Lemm~\ref{lemma_J_geometry}, the controller-level objective $\cJ(\bK_{\theta})$ is gradient dominated; however, its composition with the neural parameterization is nonconvex in $\theta$. Away from zero first-layer weights, by the chain rule, $\nabla_{\theta} \cJ(\theta) =\bJ(\theta)^\top \nabla_{\bK} \cJ(\bK_{\theta})$, where $\bJ(\theta)\in\R^{2\times 2m}$ is the Jacobian of $\theta \mapsto (\widetilde{K}_{1,\theta}, \widetilde{K}_{2,\theta})$. Since $K_{i,\theta} = K_{0} {+} \widetilde{K}_{i,\theta}$ and $K_{0}$ is constant, $\bJ(\theta)$ is also the Jacobian of $\theta \mapsto (K_{1,\theta}, K_{2,\theta})$. From ${\partial \widetilde{K}_{i,\theta}}/{\partial w_{j}} = v_{j} \ind{\{j \in \sI_{i}\}}$ and ${\partial \widetilde{K}_{i,\theta}}/{\partial v_{j}} {=} w_{j} \ind{\{j {\in} \sI_{i}\}}$, $i {\in} \{1,2\}$, and $\theta {=} [\bw^{\top} \; \bv^{\top}]^{\top}$, we get
\begin{equation}
    \bJ(\theta)
    =
    \begin{bmatrix}
        (\ind{\{\bw>0\}}\odot\bv)^\top & (\ind{\{\bw>0\}}\odot\bw)^\top \\
        (\ind{\{\bw<0\}}\odot\bv)^\top & (\ind{\{\bw<0\}}\odot\bw)^\top
    \end{bmatrix},
\end{equation}
where $\ind{\{\cdot\}}$ and $\odot$ act entrywise. Set $g {\triangleq} \nabla_{\bK} \cJ (\cdot) {=} [g_{1}, g_{2}]^{\top}$,
\begin{equation}
    \begin{gathered}
        \frac{\partial \cJ(\theta)}{\partial w_{j}}
        =
        v_{j} \bigl(g_{1} \ind{\{j \in \sI_{1}\}} + g_{2} \ind{\{j \in \sI_{2}\}} \bigr),
        \\
        \frac{\partial \cJ(\theta)}{\partial v_{j}}
        =
        w_{j} \bigl(g_{1} \ind{\{j \in \sI_{1}\}} + g_{2} \ind{\{j \in \sI_{2}\}} \bigr).
    \end{gathered}
\end{equation}
Thus, the controller-gradient component updates each neuron according to the effective gain it belongs to, and under \eqref{eq_update_rule} neurons may cross zero and move from one effective gain to the other in a single step.  Given $\theta_{k}$, write $g_{k} \triangleq \nabla_{\bK}\cJ(\bK_{\theta_{k}}) = [g_{1,k},g_{2,k}]^{\top}$ and group the weight-space quantities the analysis needs, for $i \in \{1,2\}$ and $j \in \sI$:
\begin{gather} \label{eq_weight_defs}
    \sI_{i,k} \triangleq \{j : (-1)^{i+1}w_{j,k} > 0\},
    \quad
    \sH_{i,k} \triangleq \sI_{i,k} \cap \sI_{3-i,k+1},
    \\
    \cG_{i,k} \triangleq \!\!\sum_{j\in\sI_{i,k}}\!\! (w_{j,k}^{2}{+}v_{j,k}^{2}),
    \quad
    \Xi_{k} \triangleq \smashoperator{\sum_{j\in\sH_{1,k}\cup\sH_{2,k}}} v_{j,k}^{2},
    \\
    \overline{K}_{i,\theta_{k+1}} \triangleq \!\!\sum_{j\in\sI_{i,k}}\!\! w_{j,k+1}v_{j,k+1},
    \\
\overline{\bK}_{\theta_{k+1}}\triangleq
    [\overline K_{1,\theta_{k+1}},\overline K_{2,\theta_{k+1}}]^{\top},  \quad
    \check{\bK}_{\theta_{k+1}} \triangleq \bK_{\mathrm b} + \overline{\bK}_{\theta_{k+1}}.
\end{gather}
The sets $\sI_{i,k}$ are the neurons forming the two effective gains at time $k$, and $\sH_{i,k}$ are the neurons leaving gain $i$ in one step. The controller masses form the preconditioner $D_{k} \triangleq \diag(\cG_{1,k},\cG_{2,k})$; the crossing mass $\Xi_{k}$ measures how much second-layer weight changes gains; and $\check{\bK}_{\theta_{k+1}}$ is the gain vector obtained by retaining the old group assignments. The activation-group error is defined as $\be_{k} \triangleq \bK_{\theta_{k+1}} - \check{\bK}_{\theta_{k+1}}$. Finally, $h_{k} \triangleq [\widetilde{K}_{1,\theta_{k}} g_{1,k}^{2},\, \widetilde{K}_{2,\theta_{k}} g_{2,k}^{2}]^{\top}$.

The preconditioner has the Gram representation
$D_k=\bJ(\theta_k)\bJ(\theta_k)^\top$.
When $D_k\succ0$, its leading gain update is a gradient step for the
instantaneous inner product with matrix $D_k^{-1}$.
This term depends on the neural network structure rather than on the
induced gains alone.

\begin{proposition}\label{proposition_approx_update}
    Let $\check{\bK}_{\theta_{k+1}} = \bK_{\theta_k} - \eta D_{k} g_{k} + \eta^{2} h_{k}$. Thus, $\bK_{\theta_{k+1}} = \bK_{\theta_k} - \eta D_{k} g_{k} + \eta^{2} h_{k} + \be_{k}$.
\end{proposition}
 
\begin{proof}
    If $j \in \sI_{1,k}$, then $w_{j,k+1} = w_{j,k} - \eta v_{j,k} g_{1,k}$ and $v_{j,k+1} = v_{j,k} - \eta w_{j,k} g_{1,k}$, so
    \begin{equation}
            w_{j,k+1} v_{j,k+1}
            {=}
            w_{j,k} v_{j,k}
            {-} \eta g_{1,k}(w_{j,k}^{2} {+} v_{j,k}^{2})
            {+} \eta^{2} g_{1,k}^{2} w_{j,k} v_{j,k}.
    \end{equation}
    Summing over $j \in \sI_{1,k}$, and similarly for $\sI_{2,k}$, yields
    \begin{equation}
        \overline{K}_{i, \theta_{k+1}}
        =
        \widetilde{K}_{i,\theta_k}
        -\eta \cG_{i,k} g_{i,k}
        +\eta^2 \widetilde{K}_{i,\theta_k} g_{i,k}^2,
    \end{equation}
    for $i \in \{1,2\}$. Stacking the two identities and using the definition of $\check{\bK}_{\theta_{k+1}}$ proves the first equality; the second follows from the definition of $\be_{k}$.
\end{proof}
 
\Prop~\ref{proposition_approx_update} gives a controller-level reading of policy gradient in the weight space: it acts on the effective gains as an \emph{adaptive preconditioned} policy-gradient step, with $\eta^{2}h_{k}$ the finite-step correction induced by the bilinear products $w_{j}v_{j}$, and the activation-group error $\be_{k}$.
 
\begin{remark}\label{remark_balancedness}
    Assume that $w_{j,k} \neq 0$ for all $j$. Writing $\bz_{j,k} \triangleq [w_{j,k},v_{j,k}]^{\top}$, for $j\in\sI_{i,k}$ the neuron update is $\bz_{j,k+1} = A(\varsigma_{i,k})\bz_{j,k}$ with $\varsigma_{i,k} \triangleq \eta g_{i,k}$ and $A(\varsigma)=\begin{bsmallmatrix}1&-\varsigma\\-\varsigma&1\end{bsmallmatrix}$. This yields the exact identities
    $
        w_{j,k+1}^{2} - v_{j,k+1}^{2} = (1-\varsigma_{i,k}^{2})(w_{j,k}^{2} - v_{j,k}^{2})
    $
    and
    $
        w_{j,k+1}^{2} + v_{j,k+1}^{2} = (1+\varsigma_{i,k}^{2})(w_{j,k}^{2}+v_{j,k}^{2}) - 4\varsigma_{i,k}w_{j,k}v_{j,k}
    $.
    The first is a multiplicative evolution identity for the balance quantity of gradient flow~\cite{du2018algorithmic,arora2018optimization}; exact balance is preserved, but a nonzero imbalance generally changes by the factor $1-\eta^2g_{i,k}^2$. The second describes how the neuron masses evolve. On a step that preserves the activation groups, summing the corresponding neuron-level identities gives
    \begin{equation}\label{eq_coupled_gain_mass}
       \begin{aligned}
       \widetilde K_{i,\theta_{k+1}}
          &=(1+\eta^2g_{i,k}^2)\widetilde K_{i,\theta_k}
             -\eta g_{i,k}\cG_{i,k},\\
       \cG_{i,k+1}
          &=(1+\eta^2g_{i,k}^2)\cG_{i,k}
             -4\eta g_{i,k}\widetilde K_{i,\theta_k}.
       \end{aligned}
    \end{equation}
    Thus, while the activation groups remain fixed, the two gains and two masses form a closed system; the gains alone generally do not. When groups change, the corresponding gain and mass transfer terms must also be included.
\end{remark}
 
\begin{lemma}\label{lemma_sign_stability}
    Let the initialization be as in \Lemm~\ref{lemma_initialization}. Fix $B > 0$ with $e^{B} B < 1$, set $\varpi \triangleq ({e^{B} B}/{\sqrt{1 - e^{2B} B^{2}}})\,V$, and define $\cS_{w}, \cS_{v}, \cM_{w}, \cM_{v}$ as in \eqref{eq_bad_sets}. Assume $\bK_{\theta_{t}} \in \cK_{\overline{\delta}}$ for $t=0,\dots,k$, that $\eta G_{\max} \le 1$, and that $\eta \sum_{t=0}^{k} \|g_{t}\| \le B$. Then, writing $r_{j,s} \triangleq \sqrt{w_{j,s}^{2} + v_{j,s}^{2}}$:
 
    (i) for every $s=0,\ldots,k+1$ and every neuron $j$,
    $r_{j,s} \le e^{B} r_{j,0}$ and $|w_{j,s} - w_{j,0}| \le e^{B} Br_{j,0}$;
 
    (ii) for $s = 0,\ldots, k$, $\sH_{1,s} \cup \sH_{2,s} \subseteq \cS_{w} \cup \cS_{v}$, and therefore, on $\cE_{\Xi}$, $\Xi_{s} \le \Xi_{\ast} \triangleq e^{2B} (M_{w} + M_{v})$;
 
    (iii) $\| \be_{s} \| \le 2 \sqrt{2} \, \eta\, \|g_s\| \, \Xi_{s}$ for $s = 0,\ldots, k$.
\end{lemma}
 
\begin{proof}
    For (i), recall the definition of $A(\varsigma)$ in \Rem~\ref{remark_balancedness}; the neuron update is $\bz_{j,s+1} = A(\varsigma_{i,s})\bz_{j,s}$ when $w_{j,s}\neq 0$ and $\bz_{j,s+1} = \bz_{j,s}$ otherwise. Since $\|A(\varsigma)\| = 1 + |\varsigma|$, we get $r_{j,s+1} \le ( 1 + \eta \|g_{s}\| ) r_{j,s}$, and iterating gives $r_{j,s} \le \exp( \eta \sum_{t=0}^{s-1} \|g_t\| )r_{j,0} \le e^{B} r_{j,0}$. Also $|w_{j,s} - w_{j,0}| \le \sum_{t=0}^{s-1} \eta \|g_t\|\, |v_{j,t}| \le e^{B} B r_{j,0}$.
 
    For (ii), take $j \notin \cS_{w} \cup \cS_{v}$. Then $|w_{j,0}| > \varpi / \sqrt{m}$ and $|v_{j,0}| \le V / \sqrt{m}$, so $r_{j,0} \le \sqrt{w_{j,0}^{2} + {V^{2}}/{m}}$. By the definition of $\varpi$, this implies $|w_{j,0}| > e^{B} B\, r_{j,0}$, so by (i) $w_{j,s}$ cannot cross zero for any $s\le k+1$; hence $\sH_{1,s} \cup \sH_{2,s} \subseteq \cS_{w} \cup \cS_{v}$. Consequently, on $\cE_{\Xi}$,
    \begin{equation}
        \Xi_{s} \le \smashoperator{\sum_{j \in \cS_{w} \cup \cS_{v}}} r_{j,s}^{2}
                \le e^{2B} \smashoperator{\sum_{j \in \cS_{w} \cup \cS_{v}}} r_{j,0}^{2}
                \le e^{2B} (\cM_{w} {+} \cM_{v})
                \le \Xi_{\ast}.
    \end{equation}
 
    For (iii), let $j \in \sH_{1,s}$ and write $w {=} w_{j,s}$, $v {=} v_{j,s}$, $w^{+} {=} w_{j,s+1}$, $v^{+} {=} v_{j,s+1}$ and $\varsigma \triangleq \eta g_{1,s}$, so that $w>0$, $w^{+} = w - \varsigma v < 0$ and $v^{+} = v - \varsigma w$. From $w^{+}<0$ we get $\varsigma v > w > 0$, and therefore
    \begin{equation}
        \begin{gathered}
        |w^{+}| {\le} |\varsigma||v|,
        \quad
        |w| {<} |\varsigma||v|, \\
        |v^{+}| {\le} |v| {+} |\varsigma||w| {\le} (1{+}|\varsigma|^{2})|v| {\le} 2|v|,
        \end{gathered}
    \end{equation}
    the last step because $|\varsigma| \le \eta \|g_{s}\| \le \eta G_{\max} \le 1$. Hence $|w^{+}v^{+}| \le 2|\varsigma| v^{2}$, and since neuron $j$ contributes $\be_{j,s} = [-w^{+}v^{+},\; w^{+}v^{+}]^{\top}$, $
        \|\be_{j,s}\| = \sqrt{2}\,|w^{+}v^{+}| \le 2\sqrt{2}\,\eta\,|g_{1,s}|\,v_{j,s}^{2}$.
    The case $j\in \sH_{2,s}$ is identical with $\varsigma = \eta g_{2,s}$; summing over all crossing neurons gives (iii).
\end{proof}

\Thm~\ref{theorem_main_result} combines \Lemm~\ref{lemma_sign_stability} with the initialization and step-size conditions to keep $D_k$ uniformly positive and bounded, and to control the activation-group error relative to the gradient, thereby preserving descent throughout training.

\section{INITIALIZATION AND PROOF OF THEOREM~\ref{theorem_main_result}}
\label{section_proof_main_result}

In this section, we combine the controller-level geometry of $\cJ$ with the weight-level sign-stability estimates into a single convergence statement. Recall the constants $\alpha_{\bK}>0$, $L_{\bK}$, $G_{\max}$ and $C_{\mathrm{grad}}>0$ of \Lemm~\ref{lemma_J_geometry}, all independent of $m$ and of $\beta$. Fix $0<s<S<\infty$, $V>0$, and $\varepsilon \in (0,1)$, and specialize the initialization variance to $\alpha \triangleq \beta^{-2}$. Define $\Delta_{\max} \triangleq \sigma_{\cD}^{2} P_{\max} - \cJ(\bK^{\star})$, a deterministic upper bound on the cost gap on $\cK_{\overline{\delta}}$; we choose it deterministically because $B$, and through $\varpi$ the sets $\cS_{w},\cS_{v}$ and the thresholds $M_{w},M_{v}$, are fixed before the initialization is drawn. Set
\begin{equation} \label{eq_constants}
    \begin{gathered}
        \tau {\triangleq} s\sqrt{\beta},\quad
        W {\triangleq} S\sqrt{\beta},\quad
        \ell_{s,S} {\triangleq} {s^2(\Phi(S)-\Phi(s))}/{4},
        \\
        c {\triangleq} (1{-}\varepsilon)\,p, \quad
        \lambda_{\ast} {\triangleq} {c\tau^2}/{4}, \quad
        \varrho {\triangleq} 1 - \eta \lambda_{\ast} \alpha_{\bK},
        \\
        B {\triangleq} {2\sqrt{C_{\mathrm{grad}}\Delta_{\max}}}/{(\lambda_{\ast} \alpha_{\bK})}, \qquad
        R_c {\triangleq} \sqrt{W^{2}{+}V^{2}},\;\;
        \\
        R_{\ast}^{2} {\triangleq} (1 {+} \varepsilon)(\alpha {+} \beta), \qquad    
        G_{\ast} {\triangleq} e^{2B} R_{\ast}^{2},\;\; \\
        L_{\ast} {\triangleq} \tfrac{3}{2} G_{\ast} {+} 2 \sqrt{2} \Xi_{\ast}, \qquad
        C_{\ast} {\triangleq} \tfrac{1}{2}G_{\ast} G_{\max} {+} \tfrac{L_{\bK}}{2} L_{\ast}^{2},
    \end{gathered}
\end{equation}
with $p$ as in \Lemm~\ref{lemma_initialization}. Note that $c = (1-\varepsilon)p \in (0,p)$ for every $\beta$, and that, since $\alpha = \beta^{-2}$, $\lambda_{\ast} = (1-\varepsilon)\bigl(2\Phi(V\beta)-1\bigr)\ell_{s,S}\,\beta$, which is \eqref{eq_lambda_star_thm}. The quantities in \eqref{eq_constants} are ordered so that no circularity arises: $(c,\tau) \to \lambda_{\ast} \to B \to \varpi \to (\cS_w,\cM_w,M_w) \to \Xi_{\ast} \to (L_{\ast}, C_{\ast})$.
 
For $B$ with $e^BB<1$, fix $M_w=2\E[\cM_w]$ and $M_v=2\E[\cM_v]$, so $\Xi_{\ast}=e^{2B}(M_w+M_v)$ is deterministic. Define
\begin{equation}\label{eq_boundary_constants}
    \kappa\triangleq1-\overline\delta,\quad
    \zeta\triangleq\kappa-|a+bK^{\star}|>0,\quad
    R_{\max}\triangleq r+\gamma b^2P_{\max}. \vspace{-0.4cm}
\end{equation}
The positive margin $\zeta$ follows from $\delta<\delta_0$ and is used to keep each gain inside the safe set.

\begin{definition}\label{definition_admissible_step_sizes}
    Let $e^{B} B < 1$, $e^{B} B R_{c} \le {\tau}/{2}$, $2\sqrt{2}\,\Xi_{\ast} \le {\lambda_\ast}/{4}$, and the boundary condition
    \begin{equation}\label{eq_boundary_regime}
       2\sqrt2\,|b|G_{\max}\Xi_{\ast}
          \le\slo^2 r\lambda_{\ast}\zeta.
    \end{equation}
    A step size $\eta>0$ is \emph{admissible} for such a $\beta$ if
    \begin{equation} \label{eq_admissible_step}
    \eta {\le} \eta_{0}(\beta) {\triangleq}
        \min \! \left\{ \!
        \frac{1}{G_{\max}},\frac{1}{\lambda_{\ast}\alpha_{\bK}},
        \frac{\lambda_{\ast}}{4C_{\ast}},
        \frac{d_0}{3\sigma_{\cD}^2R_{\max}G_{\ast}}
        \!\right\}.
    \end{equation}
\end{definition}
 
The third cap already ensures a positive contraction factor. Since $\lambda_{\ast}/\beta\le\ell_{s,S}<1/2\le R_{\ast}^{2}/\beta\le G_{\ast}/\beta$ we have $G_{\ast}\ge\lambda_{\ast}$; with $\alpha_{\bK}\le L_{\bK}$ from \Lemm~\ref{lemma_J_geometry} and $C_{\ast}\ge\tfrac{L_{\bK}}{2}L_{\ast}^{2}\ge\tfrac{9}{8}L_{\bK}G_{\ast}^{2}$, the third cap gives $\eta\lambda_{\ast}\alpha_{\bK}\le2/9$, hence $7/9\le\varrho<1$. The next proposition shows that these conditions can be enforced by choosing $\beta$ large enough.
 
\begin{proposition} \label{proposition_benign_width_choice}
    With the deterministic thresholds specified above, there exists $\beta_{0} > 0$ such that every $\beta \ge \beta_{0}$ satisfies the conditions of \Defn~\ref{definition_admissible_step_sizes} and $\eta_{0}(\beta)>0$, so the set of admissible step sizes $(0, \eta_{0}(\beta)]$ is nonempty. Consequently, for every $\beta \ge \beta_{0}$, every $m \ge m_{0}(\beta)\triangleq\max\{2,m_{\cE}(\beta)\}$ and every admissible $\eta$, the initialization event of \Lemm~\ref{lemma_initialization} and all deterministic bounds used below are available.
\end{proposition}
 
\begin{proof}
    Fix $\varepsilon \in (0,1)$ and $c = (1-\varepsilon)p$, so $c \in (0,p)$ for every $\beta$. We
    verify the four conditions of \Defn~\ref{definition_admissible_step_sizes} in turn as $\beta \to \infty$;
    all of them follow from the scalings
    \begin{equation} \label{eq_beta_scalings}
        \lambda_{\ast} {=} \Theta(\beta),
        \ 
        B {=} O(\beta^{-1}),
        \ 
        R_{c} {=} O(\sqrt{\beta}),
        \
        \varpi {=} O(\beta^{-1}),
    \end{equation}
    the first because $\alpha = \beta^{-2}$ gives $2\Phi(V\beta)-1 \to 1$ and hence
    $\lambda_{\ast} = (1-\varepsilon)(2\Phi(V\beta)-1)\ell_{s,S}\beta$, the others from the
    definitions of $B$, $R_{c}$ and $\varpi$.

    \emph{(a) $e^{B}B<1$.} Immediate from $B \to 0$.

    \emph{(b) $e^{B}BR_{c} \le \tau/2$.} By \eqref{eq_beta_scalings} the left side is
    $O(\beta^{-1/2}) \to 0$, while $\tau/2 = s\sqrt{\beta}/2 \to \infty$.

    \emph{(c) $2\sqrt{2}\,\Xi_{\ast} \le \lambda_{\ast}/4$.} Let $Z,U \sim \cN(0,1)$ be independent
    and write $w_{j,0} = \sqrt{\beta/m}\,Z$, $v_{j,0} = \sqrt{\alpha/m}\,U$. Since $\cM_{w}$ and
    $\cM_{v}$ are averages of $m$ i.i.d.\ copies,
    \begin{equation}
        \begin{gathered}
            \E[\cM_{w}] {=} \E\bigl[(\beta Z^{2} {+} \alpha U^{2}) \ind{ \{ |Z| \le \varpi/\sqrt{\beta} \}}\bigr],
            \\
            \E[\cM_{v}] {=} \E\bigl[(\beta Z^{2} {+} \alpha U^{2}) \ind{ \{ |U| > V/\sqrt{\alpha} \}}\bigr].
        \end{gathered}
    \end{equation}
    
    Here $\varpi/\sqrt{\beta} = O(\beta^{-3/2}) \to 0$ and $V/\sqrt{\alpha} = V\beta \to \infty$, so
    the first expectation is $O(\beta^{-7/2})$: use $\E[Z^2\ind{\{|Z|\le t\}}]=O(t^3)$ and $\Prob(|Z|\le t)=O(t)$ with $t=\varpi/\sqrt\beta$. The second vanishes by Gaussian tails at $V\beta$, even after multiplication by $\beta$. Thus $M_{\{w,v\}}=2\E[\cM_{\{w,v\}}]$ gives $\Xi_{\ast} \to 0$
    against $\lambda_{\ast} = \Theta(\beta)$.

    \emph{(d) Boundary condition.} The left side of~\eqref{eq_boundary_regime} tends to zero because $\Xi_{\ast}\to0$, whereas its right side is a fixed positive multiple of $\lambda_{\ast}=\Theta(\beta)$. Thus, it holds for all sufficiently large $\beta$. Every cap in~\eqref{eq_admissible_step} is then finite and positive, so $\eta_0(\beta)>0$. For each fixed $\beta$, \Lemm~\ref{lemma_initialization} supplies $m_0(\beta)\ge2$ and $C_0(\beta)>0$.
\end{proof}

The thresholds can be conservative because the constants
are uniform over the entire safe box. Sharper bounds could be obtained
on a smaller deterministic convex invariant region containing the
iterates and the line segments used in the proof, if such a region is
established separately.

\begin{remark} \label{remark_activation_freeze}
\Lemm~\ref{lemma_sign_stability} confines all crossing neurons to
$\cS_w\cup\cS_v$ and bounds their mass. The backbone neurons preserve
their signs; other neurons may cross, so full activation freezing is
not asserted. The preconditioner can evolve even on steps with fixed
activation groups, according to~\eqref{eq_coupled_gain_mass}.
If the groups remain fixed through iteration $k$, then
$\cG_{i,k}/\cG_{i,0}\le e^{2B}<3.109$ under $Be^B<1$.
This is an upper bound on amplification, not a lower bound on metric motion or evidence that the bound is attained.
\end{remark}

\subsection{Proof of \Thm~\ref{theorem_main_result}}
 
\begin{proof}
    Let $\bK_{k} \triangleq \bK_{\theta_{k}}$. We prove by induction that, on the event $\cE_m$, for every $t=0, 1, \dots, k$,
    \begin{equation} \label{eq_main_induction_1}
        \begin{aligned}
        &\textup{(I1)} \;\; && \bK_t \in \cK_{\overline{\delta}};
        \\
        &\textup{(I2)} && {\textstyle\sum_{j}} (w_{j,t}^{2} + v_{j,t}^{2}) \le G_{\ast};
        \\
        &\textup{(I3)} && w_{j,t} \ge \tfrac{\tau}{2\sqrt{m}}\;\, \forall j {\in} \cC_{1},
                          \;\; w_{j,t} \le -\tfrac{\tau}{2\sqrt{m}}\;\, \forall j {\in} \cC_{2};
        \\
        &\textup{(I4)} && \Xi_{t} \le \Xi_{\ast}
                          \;\;\text{and}\;\; \Delta_{t} \le \varrho^{t}\Delta_{0}.
        \end{aligned}
    \end{equation}
Cost decrease bounds the accumulated gradient, which controls neuron motion and preserves the backbone masses. These bounds control the preconditioner and activation-group error; the inward
update then preserves stability, after which smoothness gives the next
cost decrease. Although $\Xi_t$ uses the raw weight update to
$\theta_{t+1}$, that update depends only on the current gradient, so its
crossing mass can be bounded before stability is established at $t+1$.

    \emph{Base case.} For $t = 0$, the definition of $\cE_m$ gives $\bK_{\theta_0} \in \cK_{\delta}$, $|\cC_{1}|, |\cC_{2}| \ge cm$, and $\sum_{j} (w_{j,0}^{2} + v_{j,0}^{2}) \le R_{\ast}^{2} \le G_{\ast}$. Since $\overline{\delta} = \delta/2 < \delta$, $\cK_{\delta} \subseteq \cK_{\overline\delta}$ and hence $\bK_{\theta_0} \in \cK_{\overline{\delta}}$. By the definitions of $\cC_1$ and $\cC_2$, $w_{j,0} \ge {\tau}/{\sqrt{m}} \ge {\tau}/{(2\sqrt{m})}$ for $j \in \cC_{1}$, and similarly for $\cC_{2}$. Clearly $\Delta_0=\varrho^0\Delta_0$. Also $\Delta_{0} \le \Delta_{\max}$, so $\eta \|g_0\| \le \eta \sqrt{C_{\mathrm{grad}}\Delta_{\max}} \le B$, and \Lemm~\ref{lemma_sign_stability} with $k = 0$ together with $\cE_m\subseteq\cE_{\Xi}$ yields $\Xi_{0} \le \Xi_{\ast}$.

    \emph{Step 1: weight-space bounds \textup{(I2)--(I3)}.} Assume \eqref{eq_main_induction_1} holds up to time $k$. The cost decrease controls the accumulated gradient: by \eqref{eq_cgrad} and $\Delta_{t} \le \varrho^{t}\Delta_{\max}$,
    \begin{equation} \label{eq_grad_budget}
        \begin{aligned}
        \eta \sum_{t=0}^{k} \|g_{t}\|
        {\le} \eta\sqrt{C_{\mathrm{grad}}\Delta_{\max}} \sum_{t=0}^{\infty}\varrho^{t/2}
        {=} \frac{\eta\sqrt{C_{\mathrm{grad}}\Delta_{\max}}}{1-\sqrt{\varrho}}
        {\le} B,
        \end{aligned}
    \end{equation}
    the last step because $\eta\lambda_{\ast}\alpha_{\bK} \le 1$ by \eqref{eq_admissible_step} gives $1-\sqrt{\varrho} \ge (1-\varrho)/2 = \eta\lambda_{\ast}\alpha_{\bK}/2$. Invoking \Lemm~\ref{lemma_sign_stability}(i) with $s = k+1$ and summing over $j$ gives (I2):
    \begin{equation}
        {\textstyle\sum_{j}} (w_{j,k+1}^{2} + v_{j,k+1}^{2})
        \le
        e^{2B} {\textstyle\sum_{j}} (w_{j,0}^{2} + v_{j,0}^{2})
        \le
        e^{2B} R_{\ast}^{2}
        =
        G_{\ast}.
    \end{equation}

    For (I3), if $j \in \cC_{1} \cup \cC_{2}$ then $r_{j,0} \le R_{c}/\sqrt{m}$, so
    $
        |w_{j,k+1} - w_{j,0}| \le e^{B} B r_{j,0} \le {e^{B} B R_{c}}/{\sqrt{m}} \le {\tau}/{(2\sqrt{m})}
    $
    by the second inequality of \Defn~\ref{definition_admissible_step_sizes}, $e^B B R_{c} \le \tau/2$. Hence $w_{j,k+1} \ge {\tau}/{(2\sqrt{m})}$ on $\cC_{1}$ and $w_{j,k+1} \le {-\tau}/{(2\sqrt{m})}$ on $\cC_{2}$: every backbone neuron keeps its sign. The controller masses, therefore, satisfy
    \begin{equation} \label{eq_mass_lower}
        \cG_{1,k}
        \;\ge \sum_{j \in \cC_{1}} w_{j,k}^{2}
        \;\ge\; c\,m \cdot {\tau^{2}}/{(4m)}
        \;=\; {c \tau^{2}}/{4}
        \;=\; \lambda_{\ast},
    \end{equation}
    and similarly $\cG_{2,k} \ge \lambda_{\ast}$; the factors of $m$ cancel, which is the source of the width independence in \Rem~\ref{remark_width_free}. With $\cG_{1,k},\cG_{2,k} \le G_{\ast}$ from (I2), we obtain $\lambda_{\ast} I \preceq D_{k} \preceq G_{\ast} I$.

    \emph{Step 2: stability \textup{(I1)}.} By \Prop~\ref{proposition_approx_update},
    $\bK_{k+1}-\bK_k=-\eta D_kg_k+\eta^2h_k+\be_k$, and since $2|wv|\le w^2+v^2$ we have $|\widetilde K_{i,\theta_k}|\le\cG_{i,k}/2$ and hence $\|h_k\|\le(G_{\ast}/2)\|g_k\|^2$. We establish invariance using the inward direction of the update,
together with bounds on the mixing coefficient and activation-group error. Absorb the bilinear correction into
    $\widehat D_{i,k}\triangleq\cG_{i,k}-\eta\widetilde K_{i,\theta_k}g_{i,k}$, so
    $\cG_{i,k}/2\le\widehat D_{i,k}\le3\cG_{i,k}/2$.
    Write $d_{i,k}\triangleq d_i(\bK_k)$, $R_{i,k}\triangleq R_i(\bK_k)$ and $\widehat K_{i,k}\triangleq\widehat K_i(\bK_k)$ for the quantities defined below~\eqref{eq_occupancy_gradient}, so that $g_{i,k}=2d_{i,k}R_{i,k}(K_{i,k}-\widehat K_{i,k})$ and each target lies strictly inside the safe interval, $|a+b\widehat K_{i,k}|\le|a+bK^{\star}|=\kappa-\zeta$. Setting $t_{i,k}\triangleq2\eta d_{i,k}R_{i,k}\widehat D_{i,k}>0$, the exact neural update is
    \begin{equation}\label{eq_inward_update}
        \begin{aligned}
        a{+}bK_{i,k+1}{=}(1{-}t_{i,k})(a{+}bK_{i,k})
                     {+}t_{i,k}(a{+}b\widehat K_{i,k}){+}be_{i,k}.
        \end{aligned}
    \end{equation}
The occupancy and mass bounds $\sigma_i^2\le d_{i,k}\le\sigma_{\cD}^2/d_0$, $r\le R_{i,k}\le R_{\max}$ and $\lambda_{\ast}/2\le\widehat D_{i,k}\le3G_{\ast}/2$, together with the last cap in~\eqref{eq_admissible_step}, give

\vspace{-0.3cm}
    \begin{equation}
        \eta\slo^2r\lambda_{\ast}\le t_{i,k}
        \le{3\eta\sigma_{\cD}^2R_{\max}G_{\ast}}/{d_0}\le1.
    \end{equation}
\vspace{-0.5cm}
    
    Moreover, \eqref{eq_boundary_regime} and the crossing estimate imply
    $|be_{i,k}|\le2\sqrt2|b|\eta G_{\max}\Xi_{\ast}
       \le\eta\slo^2r\lambda_{\ast}\zeta\le t_{i,k}\zeta$.
    Therefore
    $|a+bK_{i,k+1}|\le(1-t_{i,k})\kappa+t_{i,k}(\kappa-\zeta)+t_{i,k}\zeta=\kappa$.
    Thus $\bK_{k+1}\in\cK_{\overline\delta}$. The accumulated gradient controls neuron motion, while this inward-pointing update establishes physical stability.

    \emph{Step 3: cost decrease \textup{(I4)}.} Combining  the inequality $\|h_k\|\le(G_{\ast}/2)\|g_k\|^{2}$ with \Lemm~\ref{lemma_sign_stability}(iii) and $\eta G_{\max}\le1$ gives
    \begin{equation}\label{eq_gain_step_bound}
        \|\bK_{k+1}-\bK_k\|\le\eta L_{\ast}\|g_k\|.
    \end{equation}
    Since $\cK_{\overline{\delta}}$ is convex and $\bK_{k}, \bK_{k+1} \in \cK_{\overline{\delta}}$, the segment joining them lies in $\cK_{\overline{\delta}}$, so $L_{\bK}$-smoothness of $\cJ$ yields $\Delta_{k+1} \le \Delta_{k} + \langle g_{k}, \bK_{k+1} - \bK_{k}\rangle + ({L_{\bK}}/{2}) \|\bK_{k+1} - \bK_{k}\|^2$. Using the decomposition of $\bK_{k+1} - \bK_{k}$ together with $D_{k} \succeq \lambda_{\ast} I$,
    \begin{equation}
        \begin{gathered}
            \langle g_{k}, \bK_{k+1} {-} \bK_{k}\rangle
            \le
            -\eta \lambda_{\ast}\|g_{k}\|^{2} + \eta^{2} \langle g_{k}, h_{k}\rangle + \langle g_{k}, \be_{k}\rangle,
            \\
            \eta^{2} \langle g_{k}, h_{k}\rangle
            \le
            \eta^{2} ({G_{\ast}}/{2}) G_{\max} \|g_{k}\|^{2},
            \\
            \langle g_{k}, \be_{k}\rangle
            \le
            2 \sqrt{2}\, \eta\, \Xi_{\ast}\, \|g_k\|^{2},
        \end{gathered}
    \end{equation}
    together with $({L_{\bK}}/{2})\|\bK_{k+1}-\bK_k\|^{2} \le ({L_{\bK}}/{2}) \eta^{2} L_{\ast}^{2} \|g_{k}\|^2$. Combining these and using $\eta C_{\ast} \le \lambda_{\ast}/4$ from \eqref{eq_admissible_step} with $2\sqrt{2}\,\Xi_{\ast} \le \lambda_{\ast}/4$,
    \begin{equation} \label{eq_descent_chain}
        \begin{aligned}
        \Delta_{k+1}
        &\le \Delta_{k} - \eta\bigl(\lambda_{\ast} - 2\sqrt{2}\,\Xi_{\ast} - \eta C_{\ast}\bigr)\|g_{k}\|^{2}
        \\
        &\le \Delta_{k} - \tfrac{\eta\lambda_{\ast}}{2}\|g_{k}\|^{2}
        \;\le\; \Delta_{k} - \eta\lambda_{\ast}\alpha_{\bK}\Delta_{k}
        \;=\; \varrho\,\Delta_{k},
        \end{aligned}
    \end{equation}
    the third inequality by gradient dominance, \Lemm~\ref{lemma_J_geometry}(iii). Hence $\Delta_{k+1} \le \varrho^{k+1}\Delta_{0}$.

    \emph{Step 4: crossing mass \textup{(I4)}, closing the induction.} Since $\Delta_{t} \le \varrho^{t}\Delta_{\max}$ now holds for $t=0,\dots,k+1$, \eqref{eq_grad_budget} with $k+1$ in place of $k$ gives $\eta\sum_{t=0}^{k+1}\|g_{t}\| \le B$, and \Lemm~\ref{lemma_sign_stability}(ii) on $\cE_m\subseteq\cE_{\Xi}$ gives $\Xi_{k+1} \le \Xi_{\ast}$. All of \eqref{eq_main_induction_1} holds at $k+1$, and the induction closes. The bound on $\|\bK_{\theta_{k}} - \bK^{\star}\|$ follows from \Lemm~\ref{lemma_J_geometry}(ii). Moreover,
$\sum_k\|\theta_{k+1}-\theta_k\|\le\sqrt{G_\ast}\,\eta\sum_k\|g_k\|<\infty$,
so the weights converge to a realization of $\bK^\star$.
\end{proof}

\section{NUMERICAL ANALYSIS}
\label{section_simulations}
In the first numerical experiment, we compare different controllers on the same problem instance. The system parameters are $a = 1.1$ and $ b = 0.1$, with initial stabilizing controller $K_{0} = -1.05$; the cost parameters are $q=2.0$, $r=0.01$, and $\gamma=0.999$; $x_0\sim \cN(0,1)$, so $\sigma_{1}^{2} = \sigma_{2}^{2} = 1/2$. Consequently, $\delta_{0} = 5\cdot 10^{-3}$.

Specifically, we have: (i) a linear controller $K^{\lin} \in \R$, as in \cite{fazel2018global}, updated by $K_{k+1}^{\lin} {=} K_{k}^{\lin} {-} \eta_{\lin} \nabla_{K} \cJ^{\lin} (K_{k}^{\lin})$; (ii) a piecewise controller $\bK^{\pw} \triangleq (K_{1}, K_{2}) \in \R^{2}$ updated directly in controller coordinates $\bK_{k+1}^{\pw} {=} \bK_{k}^{\pw} {-} \eta_{\pw} \nabla_{\bK} \cJ (\bK_{k}^{\pw})$; and (iii) ReLU controllers with different numbers of neurons $m$ updated by \eqref{eq_update_rule}. All controllers use the constant step size $\eta = \eta_{\lin} = \eta_{\pw} = 0.0005$. For the ReLU controllers, the initial weights $\theta_0$ are sampled as in \Lemm~\ref{lemma_initialization} with $\beta=10$ and $\alpha=\beta^{-2}$. The piecewise controller is initialized at $\bK_{0}^{\pw} = (K_{1, \theta_{0}}, K_{2, \theta_{0}})$, and the linear controller at $K_{0}^{\lin} = (K_{1, \theta_{0}} + K_{2, \theta_{0}})/2$, where $K_{1,\theta_0}, K_{2,\theta_0}$ correspond to the initial induced gains for the residual ReLU controller with $m = 500$.
The gradient $\nabla_{\bK}\cJ$ is evaluated exactly by differentiating $\bP(\bK) = (\bI - \gamma\bA(\bK))^{-1}\mbf(\bK)$ as in the proof of \Lemm~\ref{lemma_J_geometry}(iv); no state-trajectory sampling is involved. 
These numerical experiments complement the sufficient theoretical guarantees by examining representative parameter choices $(\beta,m,\eta)$ and illustrating the induced training dynamics.
\begin{figure}[tb!]
    \centering
    \includegraphics[width=\linewidth]{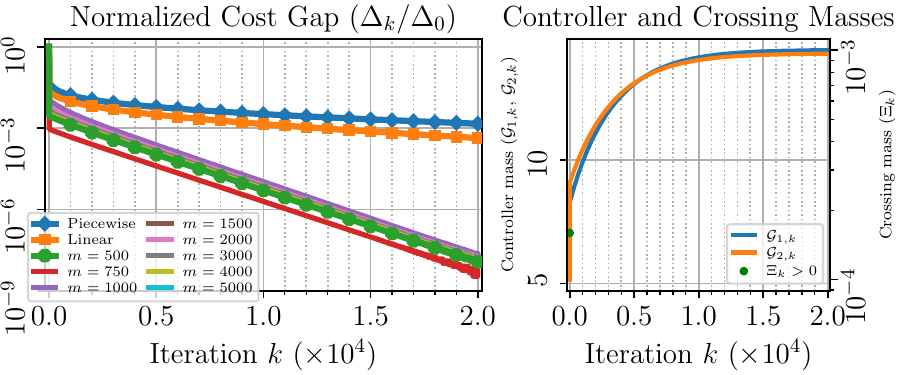}
     \caption{(Left) Cost gap divided by each method's own initial cost gap. (Right) Evolution of the controller masses ($\cG_{1,k}$, $\cG_{2,k}$) and crossing mass~($\Xi_k$). \vspace{-0.3cm}} 
     \label{fig_comparison}
\end{figure}

\figurename~\ref{fig_comparison}~(Left) shows that the normalized cost gap decreases with approximately geometric tails over the simulation horizon. Specifically, the residual ReLU controllers decrease their normalized cost gap faster than the linear and piecewise controllers at the common nominal step size, consistent with the induced scaling.
Overall, this highlights the advantages of the ReLU parameterization despite its more complex representation compared to simpler controller parameterizations.
In addition, the observed convergence of the ReLU controllers depends weakly on the number of neurons $m$. \figurename~\ref{fig_comparison}~(Right) shows the evolution of the controller masses ($\cG_{1,k}$, $\cG_{2,k}$) and crossing mass ($\Xi_k$) for the ReLU controller with $m = 500$. The controller masses satisfy $\E[\cG_{i,0}]=(\alpha+\beta)/2 \approx5$. Their evolution determines the induced scaling $D_k$, while their controlled magnitude throughout the simulation is consistent with the weight-space analysis. Moreover, $\Xi_k$ is nonzero only at $k=0$. This means that, for this ReLU controller with $m = 500$, the activation groups remain fixed for $k \ge 1$ in this realization; the error~$\be_k$ vanishes, and the remaining differences from direct-gain policy gradient are the preconditioner $D_k$ and the finite-step correction $h_k$.
\begin{figure}[tb!]
    \centering
    \includegraphics[width=\linewidth]{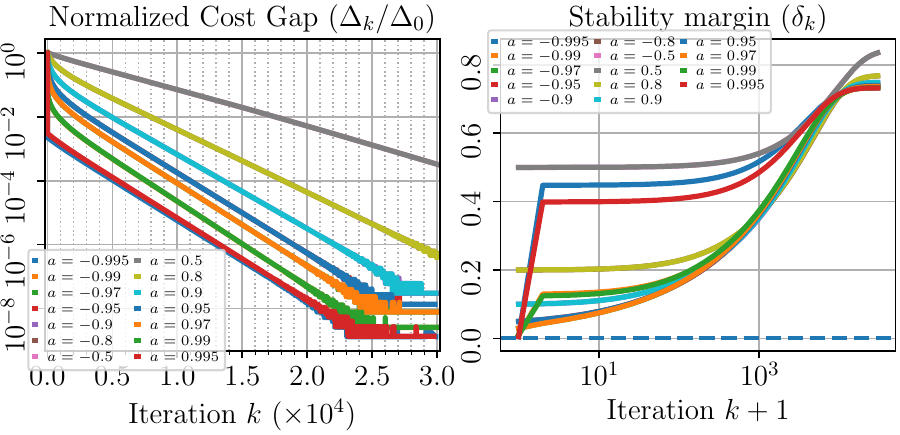}
    \caption{(Left) Normalized cost gaps for coefficients $a$ with $|a|<1$ and $K_0=0$. (Right) Stability margin evolution. \vspace{-0.3cm}}
     \label{fig_system_sweep}
\end{figure}

Second, we study the stability-margin evolution across systems with different open-loop stability margins. We vary the parameter $a$ over values satisfying $|a|<1$, with $K_0=0$; we use a residual ReLU controller with $m=500$ and the same initialization realization for all values of $a$, and set the other parameters as in the previous experiment. Varying $a$ changes the initial closed-loop stability margin, objective geometry, and optimal gain; however, \figurename~\ref{fig_system_sweep}~(Left) shows that the normalized cost gaps exhibit approximately geometric tails consistent with \Thm~\ref{theorem_main_result}. Define $\delta_{k} = \min_{i \in \{1,2\}} (1 - |a + b K_{i, \theta_{k}}|)$ as the stability margin at iteration~$k$. \figurename~\ref{fig_system_sweep}~(Right) illustrates the stability margin evolution. It shows that the residual ReLU controllers maintain a stability margin greater than $0$ over the simulation horizon for the different problem instances, consistent with \Thm~\ref{theorem_main_result}. Likewise, for an initialized ReLU controller with a small stability margin, the policy gradient quickly increases its stability margin. This is consistent with the inward-pointing gradient mechanism described in \eqref{eq_occupancy_gradient}.
\section{CONCLUSIONS}
\label{section_conclusions}

We studied the training dynamics of policy gradient optimization for deterministic scalar discounted LQR with a controller parameterized by a residual, bias-free one-hidden-layer ReLU network. We first showed that the ReLU controller can be represented by two effective gains. Second, we formalized the smoothness, quadratic growth, and gradient dominance properties of the controller cost as a function of the two effective gains on a set of $\overline{\delta}$-stable controllers, and established a uniform gradient bound on this set. Third, we proved that policy gradient on the weights induces preconditioned effective-gain dynamics, along with an explicit quadratic finite-step correction and an error due to activation-group changes. Fourth, under the stated assumptions, we identified initialization and step-size conditions such that the sequence of ReLU controllers remains stabilizing throughout training, and the corresponding cost converges geometrically to the optimal LQR cost. Finally, numerical experiments illustrated the convergence and stability of ReLU controller training by policy gradient. Future work will study training dynamics for deeper/biased neural networks, constrained higher-dimensional systems, and noisy/model-free learning. 

\bibliographystyle{IEEEtran}
\bibliography{bib/ref}

\end{document}